\documentclass[11pt]{article}
\usepackage[utf8]{inputenc}
\usepackage{hyperref}
\usepackage{amsmath,amssymb,amsthm}
\usepackage{amsfonts,tikz}
\usepackage[margin=1in]{geometry}
\usepackage{mathtools}
\usetikzlibrary{positioning,automata}
\tikzset{every loop/.style={min distance=10 mm, in=60, out=120, looseness=10}}
\newcommand{\bk}{\backslash}
\newcommand{\lm}{\lambda}
\newcommand{\mt}{\widetilde{m}}

\newcommand{\Lm}{\Lambda}
\newcommand{\tm}{\widetilde{m}}
\newcommand{\bx}{\bold{x}}
\newcommand{\by}{\bold{y}}

\newcommand{\ob}{\overline}
\newcommand{\eps}{\epsilon}
\newcommand{\sgn}{\text{sgn}}
\newcommand{\mf}{\mathcal{F}}

\newtheorem{lemma}[]{Lemma}
\newtheorem{cor}[lemma]{Corollary}
\newtheorem{theorem}[lemma]{Theorem}

\newtheorem{definition}[lemma]{Definition}

\title{Plethysms of Chromatic and Tutte Symmetric Functions}
\author{Logan Crew, Sophie Spirkl\footnote{Department of Combinatorics \& Optimization, University of Waterloo, Waterloo, ON, N2L 3E9.\newline  Emails:  lcrew@uwaterloo.ca, sspirkl@uwaterloo.ca \newline Spirkl: We acknowledge the support of the Natural Sciences and Engineering Research Council of Canada (NSERC), [funding
reference number RGPIN-2020-03912]. Cette recherche a \'et\'e financ\'ee par le Conseil de recherches en sciences naturelles et en g\'enie du Canada (CRSNG),
[num\'ero de r\'ef\'erence RGPIN-2020-03912].
}
}
\date{\today}

\begin{document}

\maketitle

\begin{abstract}

Plethysm is a fundamental operation in symmetric function theory, derived directly from its connection with representation theory. However, it does not admit a simple combinatorial interpretation, and finding coefficients of Schur function plethysms is a major open question.

In this paper, we introduce a graph-theoretic interpretation for any plethysm based on the chromatic symmetric function. We use this interpretation to give simple proofs of new and previously known plethystic identities, as well as chromatic symmetric function identities.
    
\end{abstract}

\section{Introduction}

Throughout the 1900s, the theory of symmetric functions has been greatly researched and expanded due in large part to its natural connections with representation theory. In particular, it is well-known that there exists a linear map known as the Frobenius characteristic $\mathcal{F}: CF \rightarrow \Lm$ from class functions of finite symmetric groups to symmetric functions such that for each integer partition $\lm$, the irreducible character $\chi^{\lm}$ is mapped to the Schur function $s_{\lm}$, and that this map $\mathcal{F}$ is a ring isomorphism, and in fact an isometry with respect to the classical inner products defined on the respective spaces \cite[Chapter 7.18]{stanleybook}.

In the representation theory of finite groups, there are two natural operations to combine characters $\phi$ and $\psi$ in a product-like way. One is the outer tensor product $\phi \otimes \psi$, and it is easy to verify that this corresponds under $\mathcal{F}$ to the product of symmetric functions $\mf(\phi)\mf(\psi)$. The other is composition, in the case where $\phi \circ \psi$ is defined (most often when working over representations of a general linear group to itself); this corresponds to the symmetric group operation of \emph{plethysm}, denoted $\mf(\phi)[\mf(\psi)]$.

Building from this intuition, it is possible to define the operation of plethysm in a way that extends to a definition $f[g]$ valid for any symmetric functions $f$ and $g$. This operation is fundamental to modern symmetric function theory research; many problems and results are best stated and studied through plethystic simplifications, such as questions related to LLT polynomials \cite{lolli, perLLT}, Macdonald polynomials \cite{corteel, delta}, and Hall-Littlewood polynomials \cite{hall, hall2}.

However, there is no simple combinatorial model for plethysm. Indeed, a plethysm of Schur functions is known to be Schur-positive from representation theory, but there is no known combinatorial formula for these coefficients (known as Kronecker coefficients) except in special cases, and searching for such an interpretation remains an extremely active area of research \cite{kron, kronecker}. Various combinatorial interpretations for plethysm have been suggested: Nava and Rota give a set-theoretic definition in terms of objects called \emph{partitionals} \cite{partitionals}, while Carr\'{e} and Leclerc consider crystal structures of tableaux diagrams, particularly domino tableaux \cite{carre}.

In this work, we give a new combinatorial interpretation for plethysm via graph theory using the \emph{chromatic symmetric function} $X_G$ of a graph $G$. This function was introduced by Stanley in 1995 \cite{stanley} as a symmetric function analogue of the chromatic polynomial $\chi_G$, and the study of this function has seen a resurgence in the last ten years, based largely around the Stanley-Stembridge conjecture claiming that the chromatic symmetric function of unit interval graphs is $e$-positive \cite{centered, huh, foley, newton} and the conjecture that the chromatic symmetric function distinguishes trees \cite{marked, transplanting, martin}. In \cite{delcon}, the authors extended Stanley's chromatic symmetric function $X_G$ to \emph{vertex-weighted graphs} $(G,w)$, which consist of a graph $G$ and a weight function $w: V(G) \rightarrow \mathbb{Z}^+$, as
$$
X_{(G,w)} = \sum_{\kappa \text{ proper}} \prod_{v \in V(G)} x_{\kappa(v)}^{w(v)}
$$
where the sum ranges over all proper colourings $\kappa: V(G) \rightarrow \mathbb{Z}^+$. This extended function retains basis expansions and properties of $X_G$ while admitting a simple edge deletion-contraction relation that was missing from $X_G$. It is this extended function that we will use as the basis for our model of plethysm.

Specific plethysms involving the chromatic symmetric function have naturally arisen before. In his paper originally introducing $X_G$, Stanley showed the superification relation

\newtheorem*{thm:stanley}{Theorem}

\begin{thm:stanley}[\cite{stanley}, Theorem 4.3]
$$
\omega_y(X_G)[\bx+\by] = \sum_{(\gamma, \kappa)} \prod_i x_i^{\kappa^{-1}(i)}y_i^{\kappa^{-1}(-i)}
$$
\end{thm:stanley}
where the sum ranges over ordered pairs of an acyclic orientation $\gamma$ of $G$, and all colorings $\kappa: V(G) \rightarrow \mathbb{Z}^+ \cup \mathbb{Z}^-$ such that each directed edge $u \rightarrow v$ of $\gamma$ satisfies either $\kappa(u) < \kappa(v)$, or $\kappa(u) = \kappa(v) \in \mathbb{Z}^-$. Recent work of Bernardi and Nadeau \cite{heaps} uses the theory of heaps to study the chromatic symmetric function, recovering this identity as well as others. For example, they show

\newtheorem*{thm:heaps}{Theorem}
\begin{thm:heaps}[\cite{heaps}, Theorem 5.6]

$$
\omega(X_G)[\bx+\by] = \sum_{(\gamma, \kappa_0)} p_{\lm_s(\gamma, \kappa_0)}(\bx)\prod_{i \geq 1} y_i^{\kappa^{-1}(i)}
$$

\end{thm:heaps}
 where the sum ranges over ordered pairs $(\gamma, \kappa_0)$ of an acyclic orientation $\gamma$ of $G$ and all colorings $\kappa_0: V(G) \rightarrow \mathbb{Z}_{\geq 0}$ such that each directed edge $u \rightarrow v$ of $\gamma$ satisfies $\kappa(u) \leq \kappa(v)$, and $\lm_s(\gamma, \kappa_0)$ is what they call the integer partition of \emph{source components} of $\gamma$ on $G$ restricted to $\kappa^{-1}(0)$ (more details are given in Section 4).

Our main result is a combinatorial description of any plethysm of the form $X_{(G,w)}[f]$ which admits the above results as special cases. Given $f$, rather than coloring $G$ with integers and using these to index variables arising from $f$, we simplify by coloring $V(G)$ with the variables themselves. To do so, we precisely define $Var(f)$, the set of variables described by $f$ (e.g. $Var(\bx+\by) = \{x_1,y_1,x_2,y_2,\dots\}$) and furthermore we give each variable $x$ an intrinsic sign $sgn(x)$ that is positive or negative (in particular, we treat $-x$ differently from negatively signed $x$ in the same way that $f[-\bx]$ is different from $f[\eps x]$).

\newtheorem*{thm:xpleth}{Theorem \ref{thm:xpleth}}
\begin{thm:xpleth}\label{thm:xpleth}

For $f$ an expression for which $Var(f)$ is defined, prescribe a total ordering $<$ on the elements of $Var(f)$. Then
$$
X_{(G,w)}[f] = \sum_{(\gamma, \kappa)} \prod_{v \in V(G)} \sgn({\kappa(v)}){\kappa(v)}^{w(v)}
$$
where the sum runs over all ordered pairs consisting of an acyclic orientation $\gamma$ of $G$, and a map $\kappa: V(G) \rightarrow Var(f)$ such that 
\begin{itemize}
    \item If $u \rightarrow_{\gamma} v$, then $\kappa(u) \leq \kappa(v)$ $($with respect to the total ordering on $Var(f)$$)$.
    \item Whenever $\kappa(u) = \kappa(v)$ and $sgn(\kappa(u)) = 1$,  $uv \notin E(G)$.
\end{itemize}
\end{thm:xpleth}
Since for any family $\{G_1, G_2, \dots\}$ of connected vertex-weighted graphs where $G_n$ has total weight $n$, the chromatic symmetric functions $X_{G_1}, X_{G_2}, \dots$ are algebraically independent and generate all symmetric functions \cite{bases}, our expression extends to an interpretation of any plethysm of symmetric functions. For example, $p_{\lm} = X_{G_{\lm}}$, where $G_{\lm}$ is the vertex-weighted graph with no edges and $l(\lm)$ vertices of weights $\lm_1 \geq \dots \geq \lm_{l(\lm)}$. Similarly, each element of the $e$ and $m$ bases may be recognized as (a constant multiple of) the chromatic symmetric function of a single vertex-weighted graph. Thus, for $b \in \{e,m,p\}$ we can use Theorem \ref{thm:xpleth} to directly evaluate $b_{\lm}[f]$ for any $\lm$. Then we can interpret plethysms of other symmetric functions, such as the Schur functions $s_{\lm}$, by writing $s_{\lm} = \sum_{\mu} c_{\mu}b_{\mu}$ using a known expansion of Schur functions, and evaluating $s_{\lm}[f] = (\sum c_{\mu}b_{\mu})[f] = \sum c_{\mu}(b_{\mu}[f])$ and interpreting each of these summands graph-theoretically.

We also consider plethysms involving the \emph{Tutte symmetric function}, also introduced by Stanley \cite{stanley2} and extended to vertex-weighted graphs by Aliste-Prieto, Zamora, and the authors \cite{tutteCrew} as a symmetric function with coefficient ring $\mathbb{C}[t]$ defined by
$$
XB_{(G,w)} = \sum_{\kappa} (1+t)^{e(\kappa)}\prod_{v \in V(G)} x_{\kappa(v)}^{w(v)}
$$
where the sum now ranges over \emph{all} colourings $\kappa$, and $e(\kappa)$ is the number of edges of $G$ whose endpoints receive the same colour from $\kappa$. It is equivalent to the $W$-polynomial of Noble and Welsh \cite{noble2, noble}, and serves as a natural link between the chromatic symmetric function and the $V$-polynomial of Moffatt and Monaghan \cite{vpol}.

This paper is structured as follows: in Section 2, we give background on symmetric functions, graph theory, and the chromatic symmetric function that will be necessary in our work. In Section 3, we prove Theorem \ref{thm:xpleth}, our main result. In Section 4, we give some immediate consequences of Theorem \ref{thm:xpleth}, including identities of plethysm and $X_G$ that follow. In Section 5, we extend our plethystic interpretation to the Tutte symmetric function. Finally, in Section 6 we note further directions for research.

\section{Background}

\subsection{Fundamentals of Partitions and Symmetric Functions}

 A \emph{set partition} of a set $S$ is a collection $B$ of nonempty, pairwise non-intersecting \emph{blocks} $B_1, \dots, B_k$ satisfying $B_1 \cup \dots \cup B_k = S$.
 
 An \emph{integer partition} is a tuple $\lm = (\lm_1,\dots,\lm_k)$ of positive integers such that $\lm_1 \geq \dots \geq \lm_k$. The integers $\lm_i$ are the \emph{parts} of $\lm$. If $\sum_{i=1}^k \lm_i = n$, we say that $\lm$ is a partition of $n$. The number of parts equal to $i$ in $\lm$ is given by $r_i(\lm)$. We may use simply \emph{partition} to refer to either a set or integer partition; the usage will be clear from context.
 
 If $S$ is a set with $n$ elements (so $n$ is a positive integer), we write $\pi \vdash S$ or $\lm \vdash n$ to mean respectively that $\pi$ is a set partition of $S$, and $\lm$ is a partition of $n$, and we write $|\lm| = |\pi| = n$. The number of blocks or parts is the \emph{length} of a partition, and is denoted by $l(\pi)$ or $l(\lm)$. When $\pi$ is a set partition, we will write $\lm(\pi)$ to mean the integer partition whose parts are the sizes of the blocks of $\pi$.

  A function $f(x_1,x_2,\dots) \in \mathbb{C}[[x_1,x_2,\dots]]$ is \emph{symmetric}\footnote{The choice of coefficient ring is irrelevant for the work in this paper so long as it is a field of characteristic $0$.} if $f(x_1,x_2,\dots) = f(x_{\sigma(1)},x_{\sigma(2)},\dots)$ for every permutation $\sigma$ of the positive integers $\mathbb{N}$.  The \emph{algebra of symmetric functions} $\Lambda$ is the subalgebra of $\mathbb{C}[[x_1,x_2,\dots]]$ consisting of those symmetric functions $f$ that are of bounded degree (that is, there exists a positive integer $n$ such that every monomial of $f$ has degree $\leq n$).  Furthermore, $\Lambda$ is a graded algebra, with natural grading
  $$
  \Lambda = \bigoplus_{d=0}^{\infty} \Lambda^d
  $$
  where $\Lambda^d$ consists of symmetric functions that are homogeneous of degree $d$ \cite[Chapter 7.1]{stanleybook}.

  Each $\Lambda^d$ is a finite-dimensional vector space over $\mathbb{C}$, with dimension equal to the number of partitions of $d$ (and thus, $\Lambda$ is an infinite-dimensional vector space over $\mathbb{C}$).  Some commonly-used bases of $\Lambda$ that are indexed by partitions $\lm = (\lm_1,\dots,\lm_k)$ are discussed in various textbooks (e.g. \cite[Chapter 7]{stanleybook}) and include:
\begin{itemize}
  \item The monomial symmetric functions $m_{\lm}$, defined as the sum of all distinct monomials of the form $x_{i_1}^{\lm_1} \dots x_{i_k}^{\lm_k}$ with distinct indices $i_1, \dots, i_k$.

  \item The power-sum symmetric functions, defined by the equations
  $$
  p_n = \sum_{k=1}^{\infty} x_k^n, \hspace{0.3cm} p_{\lm} = p_{\lm_1}p_{\lm_2} \dots p_{\lm_k}.
  $$
  \item The elementary symmetric functions, defined by the equations
  $$
  e_n = \sum_{i_1 < \dots < i_n} x_{i_1} \dots x_{i_n}, \hspace{0.3cm} e_{\lm} = e_{\lm_1}e_{\lm_2} \dots e_{\lm_k}.
  $$
  \item The complete homogeneous symmetric functions, defined by the equations
  $$
  h_n = \sum_{i_1 \leq \dots \leq i_n} x_{i_1} \dots x_{i_n}, \hspace{0.3cm} h_{\lm} = h_{\lm_1}h_{\lm_2} \dots h_{\lm_k}.
  $$
\end{itemize}

  We also make use of the \emph{augmented monomial symmetric functions}, defined by 
  $$
  \tm_{\lm} = \left(\prod_{i=1}^{\infty} r_i(\lm)!\right)m_{\lm}.
  $$
  
  There exists a classical involutive map $\omega: \Lambda \rightarrow \Lambda$ defined by $\omega(p_{\lm}) = (-1)^{|\lm|-l(\lm)}p_{\lm}$.
  
  \subsection{Plethysm of Symmetric Functions}
  
  All of the material in this section may be found, described in further detail, in many textbooks, e.g. \cite[Chapter 7, Appendix 2]{stanleybook}, as well as shorter expositional papers, such as \cite{expose}.
  
  The operation of \emph{plethysm} of two symmetric functions is typically denoted by $f[g]$. It is defined for all $f,g \in \Lambda$ by defining it for constants and for $p_n$, and defining how it expands across sums and products. In what follows, $c$ is an arbitrary constant, $f$ and $g$ are arbitrary symmetric functions, and the variables are implicitly the arguments of the innermost function and are omitted from the notation:
  
  \begin{itemize}
      \item $c[f] = c$.
      \item $p_n[c] = c$.
      \item $p_n[p_m] = p_{nm}$.
      \item $p_n[f+g] = p_n[f]+p_n[g]$.
      \item $p_n[fg] = p_n[f]p_n[g]$.
      \item $(fg)[h] = f[h]g[h]$.
      \item $(f+g)[h] = f[h]+g[h]$.
  \end{itemize}
  
  Note that these definitions imply that $$p_{\lm}[fg] = p_{\lm}[f]p_{\lm}[g]$$ but also $$p_{\lm}[f+g] = \prod_{i=1}^{l(\lm)} \left(p_{\lm_i}[f]+p_{\lm_i}[g]\right) \neq p_{\lm}[f]+p_{\lm}[g].$$
  
  Sometimes it is useful to modify the variables within a plethystic expression. In the literature this is most often done by writing $\bx = x_1 + x_2 + \dots$, so that $f[\bx]$ = $f[p_1(x_1, x_2, \dots)] = f(x_1,x_2,\dots)$. Using a different indeterminate $q$ we can write expressions such as
  $$
  f[q\bx] = f(qx_1, qx_2, \dots)
  $$
  so that if $f \in \Lambda^n$ we have
  $$
  f[q\bx] = q^nf[\bx].
  $$
  
  Now for a difficulty: the above rules require $p_n[-g] = -p_n[g]$. However, this would seem to disagree with the above notion of $f[q\bx]$, which would imply $p_n[-\bx] = (-1)^np_n[\bx]$. The only possible resolution is that in general \emph{plethysm does not commute with evaluation of variables}.
  
  This is a bit unintuitive, but workable. However, we are then left with a different question: how do we negate each variable separately if we cannot do so using $f[-\bx]$? The answer is that we define a special operator $\eps$ that does this for us:
  $$
  p_n[\eps \bx] = (-1)^np_n[\bx].
  $$
  
  Furthermore, it is well-known that we may use $\eps$ to give an interpretation of the plethysm of an involuted function: for $f \in \Lambda^n$ we have
 \begin{equation}\label{eq:omega}
  (\omega f)[\bx] = (-1)^nf[-\eps\bx].
  \end{equation}
  
  \subsection{Graphs}
  
  This section introduces some standard definitions in graph theory. This material can also be found with more background and illumination in standard graph theory textbooks such as \cite{diestel}.

  A \emph{graph} $G = (V,E)$ consists of a \emph{vertex set} $V$ and an \emph{edge multiset} $E$ where the elements of $E$ are (unordered) pairs of (not necessarily distinct) elements of $V$. An edge $e \in E$ that contains the same vertex twice is called a \emph{loop}.  If there are two or more edges that each contain the same two vertices, they are called \emph{multi-edges}.  A \emph{simple graph} is a graph $G = (V,E)$ in which $E$ does not contain loops or multi-edges (thus, $E \subseteq \binom{V}{2}$, the set of \emph{all} unordered pairs of vertices).  If $\{v_1,v_2\}$ is an edge, we will write it as $v_1v_2 = v_2v_1$.  The vertices $v_1$ and $v_2$ are the \emph{endpoints} of the edge $v_1v_2$. We will use $V(G)$ and $E(G)$ to denote the vertex set and edge multiset of a graph $G$, respectively.

  A \emph{subgraph} of a graph $G$ is a graph $G' = (V',E')$ where $V' \subseteq V$ and $E' \subseteq E|_{V'}$, where $E|_{V'}$ is the set of edges with both endpoints in $V'$.  An \emph{induced subgraph} of $G$ is a graph $G' = (V',E|_{V'})$ with $V' \subseteq V$.  The induced subgraph of $G$ using vertex set $V'$ will be denoted $G|_{V'}$.  A \emph{stable set} of $G$ is a subset $V' \subseteq V$ such that $E|_{V'} = \emptyset$.  A \emph{clique} of $G$ is a subset $V' \subseteq V$ such that for every pair of distinct vertices $v_1$ and $v_2$ of $V'$, $v_1v_2 \in E(G)$.

  A \emph{path} in a graph $G$ is a nonempty sequence of edges $v_1v_2$, $v_2v_3$, \dots, $v_{k-1}v_k$ such that $v_i \neq v_j$ for all $i \neq j$.  The vertices $v_1$ and $v_k$ are the \emph{endpoints} of the path. A \emph{cycle} in a graph is a nonempty sequence of distinct edges $v_1v_2$, $v_2v_3$, \dots, $v_kv_1$ such that $v_i \neq v_j$ for all $i \neq j$. Note that in a simple graph every cycle must have at least $3$ edges, although in a nonsimple graph there may be cycles of size $1$ (a loop) or $2$ (multi-edges).

  A graph $G$ is \emph{connected} if for every pair of vertices $v_1$ and $v_2$ of $G$ there is a path in $G$ with $v_1$ and $v_2$ as its endpoints.  The \emph{connected components} of $G$ are the maximal induced subgraphs of $G$ which are connected.  The number of connected components of $G$ will be denoted by $c(G)$.
  
  The \emph{complete graph} $K_n$ on $n$ vertices is the unique simple graph having all possible edges, that is, $E(K_n) = \binom{V}{2}$ where $V = V(K_n)$.
  
  An \emph{orientation} $\gamma$ of a graph $G$ is an assignment of a direction to each edge $uv \in E(G)$ by specifying either $u \rightarrow_{\gamma} v$ or $v \rightarrow_{\gamma} u$ (note that each edge of a multiedge is oriented separately). As a special case, loops are also assumed to have two possible orientations, clockwise and counterclockwise. An orientation $\gamma$ of $G$ is \emph{acyclic} if it does not have a directed cycle, meaning a cycle $v_1v_2$, $v_2v_3$, \dots, $v_kv_1$ such that $v_1 \rightarrow_{\gamma} v_2, v_2 \rightarrow_{\gamma} v_3, \dots, v_k \rightarrow_{\gamma} v_1$. Note that if $G$ has a loop it admits no acyclic orientations, and if $G$ has multiedges, any acyclic orientation of $G$ directs every edge of a given multiedge in the same direction.

  Given a graph $G$, there are two commonly used operations that produce new graphs. One is \emph{deletion}: given an edge $e \in E(G)$, the graph of $G$ \emph{with} $e$ \emph{deleted} is the graph $G' = (V(G), E(G) \bk \{e\})$, and is denoted $G \bk e$ or $G-e$. Likewise, if $S$ is a multiset of edges, we use $G \bk S$ or $G-S$ to denote the graph $(V(G),E(G) \bk S)$.

  The other operation is the \emph{contraction of an edge} $e = v_1v_2$, denoted $G / e$.  If $v_1 = v_2$ ($e$ is a loop), we define $G / e = G \bk e$.  Otherwise, we create a new vertex $v^*$, and define $G / e$ as the graph $G'$ with $V(G') = (V(G) \bk \{v_1,v_2\}) \cup v^*$, and $E(G') = (E(G) \bk E(v_1, v_2)) \cup E(v^*)$, where $E(v_1,v_2)$ is the set of edges with at least one of $v_1$ or $v_2$ as an endpoint, and $E(v^*)$ consists of each edge in $E(v_1,v_2) \bk e$ with the endpoint $v_1$ and/or $v_2$ replaced with the new vertex $v^*$.  Note that this is an operation on a (possibly nonsimple) graph that identifies two vertices while keeping and/or creating multi-edges and loops.
  
  \subsection{Vertex-Weighted Graphs and Colourings}
  
  A \emph{vertex-weighted graph} $(G,w)$ consists of a graph $G$ and a weight function $w: V(G) \rightarrow \mathbb{Z}^+$. For any $S \subseteq V(G)$ we will denote $w(S) = \sum_{v \in S} w(v)$. 
  
  Given a vertex-weighted graph $(G,w)$ and a non-loop edge $e = v_1v_2 \in E(G)$ we define in $(G,w)$ the \emph{contraction by e} to be the graph $(G/e,w/e)$, where $w/e$ is the weight function such that $(w/e)(v) = w(v)$ if $v$ is the not the contracted vertex $v^*$, and $(w/e)(v^*) = w(v_1) + w(v_2)$ (if $e$ is a loop, we define the contraction of $(G,w)$ by $e$ to be $(G \bk e, w)$).

  Let $(G,w)$ be a vertex-weighed graph (with $G$ not necessarily simple). A map $\kappa: V(G) \rightarrow \mathbb{Z}^+$ is called a \emph{colouring} of $G$. A colouring is called \emph{proper} if $\kappa(v_1) \neq \kappa(v_2)$ for all $v_1,v_2 \in V(G)$ such that there exists an edge $e = v_1v_2$ in $E(G)$. The \emph{chromatic symmetric function} was defined for unweighted graphs by Stanley \cite{stanley}; the authors extended it to vertex-weighted graphs as \cite{delcon}
  
  $$
X_{(G,w)} = \sum_{\kappa \text{ proper}} \prod_{v \in V(G)} x_{\kappa(v)}^{w(v)} 
$$
  where the sum ranges over all proper colourings $\kappa$ of $G$. This properly generalizes Stanley's definition of $X_G$ on graphs $G$ without weight functions, as when $w$ is the function assigning weight $1$ to each vertex it is easy to verify that $X_{(G,w)} = X_G$. As noted in \cite{delcon}, other expansions include
  $$
  X_{(G,w)} = \sum_{S \subseteq E(G)} (-1)^{|S|}p_{\lm(S)}
  $$
  where $\lm(S)$ is the integer partition whose parts are the total weights of the connected components of $((V(G),S), w)$, and 
  $$
  X_{(G,w)} = \sum_{\pi \text{ stable}} \tm_{\lm(\pi)}
  $$
  where the sum ranges over all (set) partitions of $V(G)$ into stable sets, with $\lm(\pi)$ the integer partition whose parts are the total weights of the parts of $\pi$. Note that if $G$ contains a loop then $X_{(G,w)} = 0$ for any $w$, and $X_{(G,w)}$ is unchanged by replacing each multi-edge of $G$ by a single edge. 
  
  The chromatic symmetric function admits the edge deletion-contraction relation\footnote{This deletion-contraction relation was used in equivalent form for the Hopf algebra of vertex-weighted graphs by Chmutov, Duzhin, and Lando \cite{chmutov}, and for the $W$-polynomial by Noble and Welsh \cite{noble}.} \cite{delcon}
\begin{equation}\label{eq:delcon}
X_{(G,w)} = X_{(G \bk e, w)} - X_{(G/e,w/e)}.
\end{equation}
  
The \emph{Tutte (or bad-colouring) symmetric function} $XB_{(G,w)}$ is defined as a symmetric function with coefficient ring $\mathbb{C}[t]$ as \cite{tutteCrew, stanley} 
  
$$
XB_{(G,w)} = \sum_{\kappa} (1+t)^{e(\kappa)}\prod_{v \in V(G)} x_{\kappa(v)}^{w(v)}
$$
where the sum ranges over \emph{all} colourings of $G$ (not just the proper ones), with $e(\kappa)$ equal to the number of edges of $G$ whose endpoints receive the same colour from $\kappa$. From \cite{tutteCrew} we also have
$$
XB_{(G,w)} = \sum_{S \subseteq E(G)} t^{|S|}p_{\lm(S)}
$$
and
$$
XB_{(G,w)} = \sum_{\pi} (1+t)^{e(\pi)}\tm_{\lm(\pi)}
$$
where this sum ranges over \emph{all} (set) partitions of $V(G)$ (not just the stable ones), letting $e(\pi)$ be the number of edges of $G$ whose endpoints lie in the same block of $\pi$.
  
Note that unlike $X_{(G,w)}$, the Tutte symmetric function is affected by multi-edges and is not annihilated by loops. Furthermore, it is easy to verify that setting $t = -1$ in the Tutte symmetric function recovers the chromatic symmetric function.

It may be shown that this function satisfies the deletion-contraction relation \cite{tutteCrew}
\begin{equation}\label{eq:tuttedelcon}
XB_{(G,w)} = XB_{(G \bk e, w)} + tXB_{(G/e,w/e)}.
\end{equation}

\section{A Plethystic Identity for the Chromatic Symmetric Function}

We will work in a ring of power series with coefficients in $\mathbb{C}$ and variables as defined below. Let a lowercase letter represent a single variable, typically $q$ or $t$. Let $\bx = x_1 + x_2 + \dots = p_1(x_1,x_2,\dots)$ represent a sum of countably many variables. We define $\by$ and any bolded letter analogously. We also define $\bx_n = x_1 + \dots + x_n$ and analogously for $\by_n$ or any other bolded letter with integer subscript. 

In the original definition of the chromatic symmetric function, a graph is colored by a map $V(G) \rightarrow \mathbb{Z}^+$, and a proper coloring gets a monomial with the corresponding variables $\{x_i : i \in \mathbb{Z}^+\}$. However, this could also be viewed by letting the monomial come straight from the coloring as a map $V(G) \rightarrow \{x_1, x_2, \dots\}$, that is, as a mapping directly from the vertex set of $G$ to an appropriate set of variables. In particular, then $X_G[\bx] = X_G(x_1, x_2, \dots)$ may be interpreted as coloring $G$ with the variables defined by the plethysm input. Theorem \ref{thm:xpleth} comes from extending this notion to any plethysm $X_G[f]$ by determining an appropriate variable set from $f$ to color with.

Before proceeding with how this variable set is defined, we highlight a few additional points. Each variable $v$ has a formal \emph{sign} $\sgn(v)$ that is $1$ or $-1$; all ``base" variables (the ones defined so far) are assumed to have positive sign. We will use an involutive, unary operator $\eps$ to change the sign of one or more variables; as noted in the introduction, the purpose of $\sgn$ is to maintain the distinction between $f[-\bx]$ and $f[\eps \bx]$. 

Sometimes we will wish to create a ``duplicate" $q'$ of a variable (or coefficient) $q$, notated by the $'$ symbol. These variables are formally different from each other, but are notated in this manner because  all variables with one or more $'$ symbols will later be substituted with the original. For example, $X_G[n] = X_G(\underbrace{1,\dots,1}_{\text{ n ones}},0,0,\dots)$ is defined by coloring the vertices of $G$ with $n$ \emph{distinguishable} copies of $1$ (terms with any $0$ will go to $0$), so we would denote these by $1,1^{'},1^{''},\dots$. \emph{Throughout this work, whenever an evaluation outside of a plethysm would use variables marked with $'$, it is assumed that they are substituted with the corresponding unmarked variable.}

When evaluating a plethystic substitution $f[g]$, the expression $g$ may contain constants, variables, addition, subtraction, multiplication\footnote{Division by expressions involving variables may be interpreted using the geometric series $\frac{1}{1-t} = 1+t+t^2+\dots$.}, and application of $\eps$. We interpret the \emph{set of variables} of $g$ in the following way:

\begin{definition}
An expression $e$ occurring inside plethystic brackets may be interpreted as a set of constants and/or variables; for clarity the corresponding set will be denoted $Var(e)$ and will be defined using the following inductive rules:

\begin{itemize}
    \item $Var(1) = \{1\}$.
    \item $Var(q) = \{q\}$.
    \item $Var(\bx_n) = \{x_1, \dots, x_n\}$.
    \item $Var(\bx) = \{x_1, x_2, \dots \}$.
    \item $Var(f+g) = Var(f) \cup Var(g) \cup (Var(f) \cap Var(g))'$ (so there are two ``copies" of each variable in $Var(f) \cap Var(g)$, with one distinguished by $'$).
     \item $Var(-f) = \{\ob{z}: z \in Var(f)\}$ where $\ob{z}$ is the same variable as $z$, but has $\sgn(\ob{z}) = -\sgn(z)$.
     \item $Var(\eps(f)) = \{-z: z \in Var(f)\}$.
     \item $Var(fg) = \{zw: z \in Var(f), w \in Var(g)\}$ (letting $sgn(zw) = sgn(z)sgn(w)$).
    \end{itemize}
\end{definition}

The modified addition rule is used because $Var(f)$ formally is a set, not a multiset. Thus we cannot at this level otherwise define expressions that require multiple copies of the same variable, like $Var(f+f)$. This is addressed by replacing the variables marked with $'$ with their originals at the time of evaluation as noted above. For example, $Var(\bx+\bx_3) = \{x_1,x_2,x_3,x_1^{'},x_2^{'},x_3^{'},x_4,x_5,\dots\}$. More generally, when $f$ is a symmetric function we can compute its variable set easily by the addition rule and multiplication rules. For example, $Var(e_2(x_1,x_2,\dots)) = Var(\sum_{i < j} x_ix_j) = \cup_{i < j} Var(x_ix_j) = \{x_1x_2, x_1x_3, \dots\}$

Note carefully the difference between $Var(-f)$ and $Var(\eps(f))$: the former changes the inherent sign of the variable, whereas the latter changes the variable itself. For example, $Var(-\bx_2) = \{\overline{x_1},\overline{x_2}\}$ where the overline indicates that only the sign of the variable has changed. However, $Var(\eps(\bx_2)) = \{-x_1,-x_2\}$. This difference will be more carefully illuminated shortly.

We recall here the identity \cite[Lemma 3]{delcon}
\begin{equation}\label{eq:pbasis}
X_{(G,w)} = \sum_{S \subseteq E(G)} (-1)^{|S|}p_{\lm(S)}
\end{equation}
as it will be important in what follows.

\begin{theorem}\label{thm:xpleth}

For $f$ an expression for which $Var(f)$ is defined, prescribe a total ordering $<$ on the elements of $Var(f)$. Then
$$
X_{(G,w)}[f] = \sum_{(\gamma, \kappa)} \prod_{v \in V(G)} \sgn({\kappa(v)}){\kappa(v)}^{w(v)}
$$
where the sum runs over all ordered pairs consisting of an acyclic orientation $\gamma$ of $G$, and a map $\kappa: V(G) \rightarrow Var(f)$ such that 
\begin{itemize}
    \item If $u \rightarrow_{\gamma} v$, then $\kappa(u) \leq \kappa(v)$ $($with respect to the total ordering on $Var(f))$.
    \item Whenever $\kappa(u) = \kappa(v)$ and $sgn(\kappa(u)) = 1$,  $uv \notin E(G)$.
\end{itemize}
\end{theorem} 

Here we can see clearly the difference between changing $\kappa(v) = z$ to one of $\kappa(v) = \ob{z}$ or $\kappa(v) = -z$: the former changes the sign of the corresponding term in the sum by $-1$, while the latter changes the sign by $(-1)^{w(v)}$.

Before we prove the formula, some examples might help to clarify how to evaluate it. $X_{(G,w)}[1]$ will be the sum over all $(\gamma, \kappa)$, where the map $\kappa$ assigns the variable $1$ to each vertex of $G$. Since two vertices with the same positive colour cannot be adjacent, it follows that if $G$ has any edges $X_{(G,w)}[1] = 0$. Otherwise, there is only the trivial colouring, and the product of all the $1$s is $1$, so if $G$ has no edges $X_{(G,w)}[1] = 1$.

If we evaluate $X_{(G,w)}[\bx]$, then the colouring is required to be simply a usual proper colouring of $G$ with the variables $x_1, \dots$, and the acyclic orientation $\gamma$ corresponding to $\kappa$ is uniquely defined, so we have that our definition for $X_{(G,w)}[\bx]$ coincides with $X_{(G,w)}(x_1,\dots)$.

To evaluate $X_{(G,w)}[-1]$, there is only one colouring, which gives every vertex of $G$ the variable $1$ with a negative sign. Since edges between two vertices with equal negative colour can be in either direction, we get a product of $1$ for each choice of acyclic orientation of $G$, and an overall sign of $(-1)^{|V(G)|}$, so we recover a classic formula of Stanley's for $\chi_G(-1)$ \cite{acyclic}.

To compute $X_{(G,w)}[n] = X_{(G,w)}[1+\dots+1]$, we colour the vertices of $G$ with $n$ distinguished copies of $1$. Since all ``variables" are positive, the colouring must be proper, and the corresponding acyclic orientation is uniquely defined, so we simply get a product of $1$ for each proper $n$-colouring of $G$, so as expected $X_{(G,w)}[n] = \chi_G(n)$.

\begin{proof}

First, note that for all $(G,w)$, as shown above $X_{(G,w)}[1]$ is clearly equal to $1$ if the graph has no edges and $0$ otherwise, $X_{(G,w)}[\bx] = X_{(G,w)}(x_1,\dots)$, and both of these agree with both sides of the above equation for all graphs.
We proceed by induction twice.

The overall induction will be on the number of edges of a graph, using deletion-contraction. To establish the base case of this overall induction, we will induct on the complexity of the expression inside plethystic brackets.

First, the overall base case is graphs with no edges. For a partition $\lm$ with parts $\lm_1 \geq \dots \geq \lm_k$, let $(G_{\lm},w_{\lm})$ be the graph with no edges and $l(\lm)$ vertices with weights prescribed by the parts of $\lm$, and note that $X_{(G_{\lm}, w_{\lm})} = p_{\lm}$. Give an ordering $v_1, \dots, v_k$ to the vertices of $(G_{\lm},w_{\lm})$ such that $w(v_i) = \lm_i$. 

The desired result is clear for the base expressions $1$ and $\bx$, so we must show by induction that the result holds for more complex expressions. Since we only have the empty acyclic orientation, we suppress $\gamma$ from the notation in this portion of the proof. Suppose that for expressions $f$ and $g$ we have
$$
X_{(G_{\lm},w_{\lm})}[f] = p_{\lm}[f] = \sum_{\kappa: V(G) \rightarrow Var(f)} \prod_{v \in V(G)} \sgn(\kappa(v))\kappa(v)^{w(v)}
$$
and likewise for $g$.

First, addition. We start by showing that under our definition have
$$p_{\lm}[f+g] = \prod_{i=1}^{l(\lm)}(p_{\lm_i}[f]+p_{\lm_i}[g]) = $$ $$ \prod_{i=1}^{l(\lm)} \left(\sum_{\kappa: \{v_i\} \rightarrow Var(f)} \sgn(\kappa(v))\kappa(v)^{w(v)} + \sum_{\kappa: \{v_i\} \rightarrow Var(g)} \sgn(\kappa(v))\kappa(v)^{w(v)} \right) = $$
$$
\prod_{i=1}^{l(\lm)} \left(\sum_{\kappa: \{v_i\} \rightarrow Var(f) \sqcup Var(g) \sqcup (Var(f) \cap Var(g))'} \sgn(\kappa(v))\kappa(v)^{w(v)} \right) = $$ $$\sum_{\kappa: V(G) \rightarrow Var(f) \sqcup Var(g) \sqcup (Var(f) \cap Var(g))'} \prod_{v \in V(G)} \sgn(\kappa(v))\kappa(v)^{w(v)}.
$$

Next, negation. By rules of plethysm we have
$$
p_{\lm}[-f] = (-1)^{|\lm|}(\omega p_{\lm})[f] = (-1)^{|\lm|}(-1)^{|\lm|-l(\lm)} p_{\lm}[f] = 
$$
$$
(-1)^{l(\lm)}\sum_{\kappa: V(G) \rightarrow Var(f)} \prod_{v \in V(G)} \sgn(\kappa(v))\kappa(v)^{w(v)} = \sum_{\kappa: V(G) \rightarrow Var(f)} \prod_{v \in V(G)} (-\sgn(\kappa(v)))\kappa(v)^{w(v)} = 
$$
$$
\sum_{\kappa: V(G) \rightarrow \ob{Var(f)}} \prod_{v \in V(G)} \sgn(\kappa(v))\kappa(v)^{w(v)}
$$
We continue to $\eps$. We have
$$
p_{\lm}[\eps f] = (-1)^{|\lm|}p_{\lm}[f] = \sum_{\kappa: V(G) \rightarrow Var(f)} \prod_{v \in V(G)} \sgn(\kappa(v))(-\kappa(v))^{w(v)} = $$ $$\sum_{\kappa: V(G) \rightarrow Var(\eps (f))} \prod_{v \in V(G)} \sgn(\kappa(v))\kappa(v)^{w(v)}
$$

Finally, for multiplication we have
$$
p_{\lm}[fg] = p_{\lm}[f]p_{\lm}[g] = $$ $$ \left(\sum_{\kappa: V(G) \rightarrow Var(f)} \prod_{v \in V(G)} \sgn(\kappa(v))\kappa(v)^{w(v)}\right)\left(\sum_{\kappa: V(G) \rightarrow Var(g)} \prod_{v \in V(G)} \sgn(\kappa(v))\kappa(v)^{w(v)}\right) = 
$$
$$
\sum_{\kappa: V(G) \rightarrow Var(f) \times Var(g)} \prod_{v \in V(G)} \sgn(\kappa(v))\kappa(v)^{w(v)}.
$$

We now proceed to the main inductive step. We assume that the claim has been shown for all vertex-weighted graphs of less than $k$ edges for some positive integer $k$, and we prove that the claim holds for vertex-weighted graphs with $k$ edges. Let $(G,w)$ be an arbitrary vertex-weighted graph with $k$ edges, let $f$ be arbitrary, and select any edge $e$ of $G$. We wish to establish the result for $X_{(G,w)}[f]$. For brevity, we will denote $\prod_{v \in V(G)} \sgn(\kappa(v))\kappa(v)^{w(v)}$ by $x_{\kappa}$.

Starting with the deletion-contraction relation
$$
X_{(G,w)}[f] = X_{(G-e,w)}[f]-X_{(G/e,w/e)}[f],
$$
using the inductive hypothesis on $G-e$ and $G/e$, and rearranging it suffices to show that
\begin{equation}\label{eq:sub}
\sum_{\substack{(\gamma, \kappa) \\ V(G-e)}} x_{\kappa} = \sum_{\substack{(\gamma, \kappa) \\ V(G)}} x_{\kappa}+\sum_{\substack{(\gamma, \kappa) \\ V(G/e)}} x_{\kappa}.
\end{equation}
We proceed in a manner essentially identical to the proof of \cite[Theorem 8]{delcon}. For a fixed choice of $\gamma$ and $\kappa$ in $G-e$, we let $a_{G-e}(\gamma,\kappa)$ denote the corresponding summand of the first sum in \eqref{eq:sub}, and we define
      \begin{itemize}\label{it:acyc}
      \item $\gamma_1$ is the orientation of $(G,w)$ with $v_1 \rightarrow v_2$ and all other edges oriented as in $\gamma$.
      \item $\gamma_2$ is the orientation of $(G,w)$ with $v_2 \rightarrow v_1$ and all other edges oriented as in $\gamma$.
      \item If $\kappa(v_1) = \kappa(v_2)$, $\kappa_e$ is the colouring of $(G / e, w / e)$ with $\kappa_e(v^*) = \kappa(v_1)$ where $v^*$ is the vertex created by the contraction of $e$, and for all other vertices $v$, $\kappa_e(v) = \kappa(v)$.
      \item If $\kappa(v_1) \neq \kappa(v_2)$, then $\kappa_e$ is not defined.
      \end{itemize}
       We now define $a_G(\gamma_1,\kappa)$, $a_G(\gamma_2, \kappa)$, and $a_{G/e}(\gamma / e, \kappa_e)$ analogously for all orientations $\gamma$ and all colourings $\kappa$ to be equal to the corresponding summands from sums in \eqref{eq:sub}, or $0$ if no such summand exists (if $\gamma$ is not acyclic in its graph, or if $\kappa$ is undefined or is not a valid colouring of $G$ with respect to $\gamma$). Using these definitions, to show \eqref{eq:sub} it suffices to show the stronger statement that for every acyclic orientation $\gamma$ of $G \bk e$, and every valid $\kappa$ of $(G \bk e, w)$ with respect to $\gamma$, we have
      \begin{equation}\label{eq:invkappa}
a_{G-e}(\gamma,\kappa) = a_G(\gamma_1, \kappa)+a_G(\gamma_2, \kappa)+a_{G/e}(\gamma / e, \kappa_e)
      \end{equation}
      since it is easy to verify that every summand of each sum of \eqref{eq:sub} is counted exactly once in this way.
      
      We split into cases based on whether $\gamma$ has a directed path between $v_1$ and $v_2$ (note that it does not contain both a path from $v_1$ to $v_2$ and one from $v_2$ to $v_1$ since then $\gamma$ would contain an oriented cycle). Suppose that $\gamma$ contains such a path; without loss of generality we may assume it is from $v_1$ to $v_2$. Then $\gamma_2$ and $\gamma_e$ both contain oriented cycles in their respective graphs, and so the corresponding terms of \eqref{eq:invkappa} are $0$. However, $\gamma_1$ does not contain an oriented cycle of $(G, w)$. Clearly $\kappa(v_1) \leq \kappa(v_2)$. We split into subcases:
      \begin{itemize}
          \item If $\kappa(v_1) < \kappa(v_2)$, or if $\kappa(v_1) = \kappa(v_2)$ is negative, then the remaining two terms are both valid, and $a_{G-e}(\gamma,\kappa) = a_G(\gamma_1, \kappa)$.
          \item If $\kappa(v_1) = \kappa(v_2)$ is positive, then every vertex on the direct path from $v_1$ to $v_2$ has the same positive colour, which is a contradiction since we start with $\gamma$ and $\kappa$ that are valid for $G-e$. Thus, this case cannot occur.
      \end{itemize}

      Now assume that there is no directed path. Then all of $\gamma_1$, $\gamma_2$, and $\gamma_e$ are acyclic orientations.  We split into subcases based on $\kappa$.  
      \begin{itemize}
          \item If $\kappa(v_1) \neq \kappa(v_2)$, then without loss of generality consider when $\kappa(v_1) < \kappa(v_2)$. Then the terms for $\gamma_1$ and $\gamma_e$ are zero, and the other two are equal.
          \item If $\kappa(v_1) = \kappa(v_2)$ is positive, then the terms for $\gamma_1$ and $\gamma_2$ are zero, and the other two are equal.
          \item If $\kappa(v_1) = \kappa(v_2)$ are negative, then all terms are nonzero.  However, $a_{G/e}(\gamma / e, \kappa_e)$ has the opposite sign of the others since it has one fewer negative vertex, and the conclusion holds.
      \end{itemize}
      
      This finishes the inductive step and concludes the proof.

\end{proof}

\section{Consequences of the Plethystic Identity}

One application of Theorem \ref{thm:xpleth} is to provide a direct combinatorial interpretation of an otherwise abstract, algebraic concept. 

Using this identity, and the fact that many classical symmetric function bases may be written as chromatic symmetric functions of vertex-weighted graphs, we can provide novel proofs for old and new plethystic identities of symmetric functions. We first establish the following result, whose unweighted version was proved in a related way by Bernardi and Nadeau using heaps\footnote{This is a vertex-weighted analogue of the coproduct in the Hopf algebra of graphs.}:

\begin{lemma}\label{lem:graphadd}[\cite{heaps}, proof of Theorem 5.6]

$$
X_{(G,w)}[f+g] = \sum_{A \sqcup B = V(G)} X_{(G|_A, w|_A)}[f]X_{(G|_B,w|_B)}[g]
$$
where $A$ and $B$ form a partition of $V(G)$, and $w|_A$ is the restriction of the weight function $w$ to only those vertices in $A$.
\end{lemma}

\begin{proof}
We wish to show that
$$
\sum_{\substack{\gamma \text{ acyclic on } G \\ \kappa:V(G) \rightarrow Var(f) \sqcup Var(g) \sqcup (Var(f) \cap Var(g))'}} \prod_{v \in V(G)} \sgn(\kappa(v))\kappa(v)^{w(v)} = $$ 
$$ \sum_{A \sqcup B = V(G)} \left(\sum_{\substack{\gamma \text{ acyclic on } G|_A \\ \kappa:A \rightarrow Var(f)}} \prod_{v \in A} \sgn(\kappa(v))\kappa(v)^{w(v)}\right)\left(\sum_{\substack{\gamma \text{ acyclic on } G|_B \\ \kappa:B \rightarrow Var(g)}} \prod_{v \in B} \sgn(\kappa(v))\kappa(v)^{w(v)}\right).
$$

It suffices to show that there is a weight-preserving bijection between 
\begin{enumerate}
    \item Ordered pairs $(\gamma, \kappa)$ for $G$.
    \item A choice of $A \sqcup B = V(G)$, a pair $(\gamma_A, \kappa_A)$ for $G|_A$, and a pair $(\gamma_B, \kappa_B)$ for $G|_B$.
\end{enumerate}

Recall that each pictured sum depends on some fixed total ordering of the variable set, but is independent of the choice. Thus, we may choose our total orderings so that the ordering on $Var(f) \sqcup Var(g) \sqcup (Var(f) \cap Var(g))'$ is an extension of the orderings on $Var(f)$ and $Var(g)$. For this purpose we interpret that the elements of $Var(f) \cap Var(g)$ use their induced ordering from $Var(f)$, and the elements of $(Var(f) \cap Var(g))'$ use their induced ordering from $Var(g)$.

Given this, constructing the bijection is straightforward.  Going from $1$ to $2$, clearly a colouring of $V(G)$ induces a choice of $A$ and $B$ (and the colourings on them) sorting by which vertices get a colour from $Var(f)$ and which from $Var(g)$, with the convention that vertices coloured from $Var(f) \cap Var(g)$ are placed in $A$, and vertices  coloured from $(Var(f) \cap Var(g))'$ are placed in $B$. Furthermore, $\gamma$ induces the appropriate orientations $\gamma_A$ and $\gamma_B$.

Conversely, from $2$ to $1$, the choice of $A$, $B$, $\kappa_A$, and $\kappa_B$ induces a unique $\kappa$ for $V(G)$.  Furthermore, there is a unique way to construct an appropriate $\gamma$. First, orient all edges that are defined by $\gamma_A$ or $\gamma_B$. All remaining edges are between $A$ and $B$, and so are between vertices that receive different colours, and so their orientation is uniquely determined by the total ordering of $Var(f) \sqcup Var(g)$. This demonstrates the bijection and completes the proof.

\end{proof}

Using Lemma \ref{lem:graphadd}, we can establish classical plethystic identities:

\begin{cor}[Folklore]
Given any arithmetic expressions $f$ and $g$, the following identities hold:

    $$
    e_n[f+g] = \sum_{i=0}^n e_i[f]e_{n-i}[g]
    $$
    $$
    h_n[f+g] = \sum_{i=0}^n h_i[f]h_{n-i}[g]
    $$
\end{cor}
\begin{proof}
    Applying Lemma \ref{lem:graphadd} to the complete graph $K_n$ we compute
    $$
    X_{K_n}[f+g] = \sum_{A \sqcup B = V(K_n)} X_{K_n|_A}[f]X_{K_n|_B}[g] = \sum_{A \sqcup B = V(K_n)} X_{K_{|A|}}[f]X_{K_{|B|}}[g]
    $$
    
    The number of occurrences of the summand $X_{K_{|A|}}[f]X_{K_{|B|}}[g]$ will be $\binom{n}{|A|}$, so the equality reduces to
    $$
    X_{K_n}[f+g] = \sum_{i=0}^n \binom{n}{i}X_{K_i}[f]X_{K_{n-i}}[g].
    $$
    As $X_{K_n} = n!e_n$, we further get 
    $$
    n!e_n[f+g] = \sum_{i=0}^n \binom{n}{i}i!e_i(n-i)!e_{n-i}
    $$
    and simplifying this equation establishes the first identity.  The second then follows from an analogous argument and noting that
    $$
    X_{K_n}[-f] = (-1)^nn!h_n[f].
    $$

\end{proof}
We can also easily establish an analogous identity:
\begin{cor}
    $$
    \mt_{\lm}[f+g] = \sum_{\mu \subseteq \lm} \mt_{\mu}[f]\mt_{\lm 
    \bk \mu}[g]
    $$
    where $\mu \subseteq \lm$ means that the parts of $\mu$ are a submultiset of the parts of $\lm$, and $\lm \bk \mu$ is the partition whose parts are those in $\lm$ and not in $\mu$.
\end{cor}

\begin{proof}
This follows immediately from applying Lemma \ref{lem:graphadd} to the vertex-weighted complete graph with weights $\lm_1 \geq \dots \geq \lm_{l(\lm)}$.

\end{proof}

The graph-based plethystic perspective also forms a unifying framework for a number of related identities of the chromatic symmetric function. In particular, our construction directly implies all of the following results:

\begin{cor}\label{cor:delcon}[\cite{delcon}, Theorem 4]

$$
(-1)^{w(G)-|V(G)|}\omega(X_{(G,w)})[\bx] = \sum_{(\gamma, \kappa)} \prod_{v \in V(G)} x_{\kappa(v)}
$$
where the sum runs over all ordered pairs $(\gamma, \kappa)$ where $\gamma$ is an acyclic orientation of $G$ and $\kappa: V(G) \rightarrow \mathbb{Z}^+$ satisfies that if $u \rightarrow_{\gamma} v$ then $\kappa(u) \leq \kappa(v)$.
\end{cor}

\begin{proof}

Using \eqref{eq:omega} we have
$$
\omega(X_{(G,w)})[\bx] = (-1)^{w(G)}X_{(G,w)}[-\eps \bx] = (-1)^{w(G)-|V(G)|}X_{(G,w)}[-\bx]
$$
and now the result follows from Theorem \ref{thm:xpleth}.

\end{proof}

\begin{cor}[A vertex-weighted generalization of \cite{stanley}, Theorem 4.3]\label{thm:stanley}
Given a symmetric function $f$ in variables $\bx$ and $\by$, let $\omega_y(f)$ denote the application of the involution $\omega$ to only the $y$ variables (specifically, apply $\omega$ treating the function as a symmetric function in $\by$ with coefficients in $\mathbb{C}[\bx]$). Then
$$
\omega_y(X_{(G,w)})[\bx+\by] = \sum_{(\gamma, \kappa)} (-1)^{w^{-}(G)-|V^{-}(G)|} \prod_i x_i^{\kappa^{-1}(i)}y_i^{\kappa^{-1}(-i)}
$$
where 
\begin{itemize}
    \item The sum runs over all $(\gamma, \kappa)$ where $\gamma$ is an acyclic orientation of $G$ and $\kappa: V(G) \rightarrow \mathbb{Z} \bk \{0\}$ is a map such that if $u \rightarrow_{\gamma} v$, then $\kappa(u) \leq \kappa(v)$, with equality if and only if $\kappa(u) = \kappa(v)$ is negative.
    \item $V^{-}(G)$ is the set of vertices receiving a negative colour from $\kappa$, and $w^{-}(G)$ is the total weight of those vertices.
\end{itemize} 
\end{cor}

Note: the function $\omega_y(f)[\bx+\by]$ is called the \emph{superification} of $f$ in the literature.

\begin{proof}

We have
$$
\omega_y(X_{(G,w)})[\bx+\by] = X_{(G,w)}[\bx-\eps \by]
$$
and the proof now follows analogously to the proof of Corollary \ref{cor:delcon} after identifying the colour $i$ with the variable $x_i$ and the colour $-i$ with the variable $-y_i$. 

\end{proof}

In \cite{heaps}, Bernardi and Nadeau found a new interpretation for the chromatic symmetric function of a graph. Given a graph $G$, label the vertices $1, 2, \dots, |V(G)|$ arbitrarily. Given an acyclic orientation $\gamma$ of $G$, define the \emph{source components} of $G$ with respect to $\gamma$ as follows: the first component $S_1$ consists of all vertices that can be reached via a directed path from the vertex labelled $1$. Then, delete these vertices, and let $S_2$ consist of all vertices in $G \bk S_1$ that can be reached via a directed path from $\min(v | v \in V(G)-S_1)$. Then $S_3, \dots$ are defined inductively in a similar way until all vertices are in a component. Then  $\lm_s(G, \gamma)$ is the integer partition whose parts are the sizes of the source components. 

Bernardi and Nadeau gave an alternate expansion for the chromatic symmetric function that generalizes easily to vertex-weighted graphs:
\begin{lemma}[\cite{heaps}, Proposition 5.3]\label{lem:source}

$$\omega(X_{(G,w)}) = \sum_{\gamma} p_{\lm_s((G,w) \gamma)}$$

where the sum ranges over all acyclic orientations $\gamma$ of $G$, and $\lm_s((G,w),\gamma$ is the partition whose parts are the total weights of the source components of $\gamma$.

\end{lemma}

\begin{proof}
The proof given in \cite{heaps} accommodates the incorporation of vertex weights with no change. 
\end{proof}

Then their plethystic result on source components follows from Theorem \ref{thm:xpleth}:

\begin{cor}[\cite{heaps}, Theorem 5.6]\label{thm:heaps}

$$
\omega(X_{(G,w)})[\bx+\by] = \sum_{(\gamma, \kappa_0)} p_{\lm_s(\gamma_0, \kappa_0)}(\bx)\prod_{i \geq 1} y_i^{\kappa^{-1}(i)}
$$

where the sum ranges over ordered pairs $(\gamma, \kappa_0)$ of an acyclic orientation $\gamma$ of $G$ and all colorings $\kappa_0: V(G) \rightarrow \mathbb{Z}_{\geq 0}$ such that each directed edge $u \rightarrow v$ of $\gamma$ satisfies $\kappa(u) \leq \kappa(v)$, and $\lm_s(\gamma_0, \kappa_0)$ is the integer partition whose parts are the total weights of the source components of $\gamma$ on $G$ restricted to $\kappa_0^{-1}(0)$.

\end{cor}

\begin{proof}

As in the proof of Lemma \ref{lem:graphadd} we have
$$
\omega(X_{(G,w)})[\bx+\by] = \sum_{A \sqcup B = V(G)} \omega(X_{(G|_A, w|_A)})[\bx]\omega(X_{(G|_B,w|_B)})[\by].
$$
where we apply an appropriate total ordering to the variables.

But this is equivalent to the proposed formula by interpreting the vertices colored $0$ as $A$, all other vertices as $B$, and applying Lemma \ref{lem:source} to $A$ and Theorem \ref{thm:xpleth} to $B$.

\end{proof}

\begin{cor}
$$
X_{(G,w)}[2\bx] = \sum_{\substack{\pi \vdash V(G) \\ C(\pi) \text{ are bipartite}}} 2^{|C(\pi)|}\mt_{\lm(\pi)}
$$
where $C(\pi)$ is the set of connected components of $\pi$ (and a graph is bipartite if it has a proper colouring using at most two colours).
\end{cor}

\begin{proof}

We have 
$$
X_{(G,w)}[2\bx] = X_{(G,w)}[\bx+\bx] = \sum_{\kappa} \prod_{v \in V(G)} \kappa(v)^{w(v)}
$$
where $\kappa: V(G) \rightarrow \{x_1,x_2,\dots\} \sqcup \{x_1',x_2',\dots\}$ is a proper colouring.

For each $\pi \vdash V(G)$, let $\pi_{\kappa}$ denote the set of colourings $\kappa$ such that for each block $B$ of $\pi$, we have (post-substitution) $\kappa(v) = \kappa(w)$ for all $v, w \in B$ (so for some $i$, every vertex in the block is coloured $x_i$ or $x_i'$). This sum may then be further striated as
$$
X_{(G,w)}[2\bx] = \sum_{\pi \vdash V(G)} \sum_{\kappa \in \pi_{\kappa}} \prod_{v \in V(G)} \kappa(v)^{w(v)}
$$

Letting $\lm(\pi)$ denote the integer partition with parts given by the total weights of the blocks of $\pi$, by symmetry it suffices to find the number of monomials of the form $x_1^{\lm_1} \dots x_{l(\lm)}^{\lm_{l(\lm)}}$.
First, note that $\pi$ admits no valid proper colourings if any connected component of $\pi$ is not bipartite.  In the event that each connected component of $\pi$ is bipartite, we first choose a variable for each component, which can be done in $\prod_i r_i(\lm)!$ ways. Then, each component may be filled with its assigned post-substitution variable $x_i$ in exactly two ways, since by picking some fixed vertex in the component and colouring it with either $x_i$ or $x_i'$, all other vertices in the component have their choice uniquely determined. Thus the sum evaluates as
$$
X_{(G,w)}[2\bx] = \sum_{\substack{\pi \vdash V(G) \\ C(\pi) \text{ are bipartite}}} 2^{|C(\pi)|}\prod_i (r_i(\lm)!)m_{\lm(\pi)}
$$
and the conclusion follows.
\end{proof}

\section{A Plethystic Formulation for the Tutte Symmetric Function}

To extend our plethystic construction to the Tutte symmetric function, we introduce generalized versions of the graph-theoretic objects we need. Define a \emph{biorientation} $\gamma$ to assign to each edge either a single direction, or both directions simultaneously. Edges that point in both directions will be called \emph{bidirected} edges. Let a \emph{bicycle} of a bioriented graph be a cycle containing only bidirected edges. We define a biorientation $\gamma$ to be acyclic if it contains no directed cycle \emph{that uses at least one singly directed edge}. Thus, the only directed cycles allowed are bicycles. 

In what follows, we evaluate $XB_{(G,w)}[f]$. Given a vertex-weighted graph $(G,w)$, a biorentation $\gamma$ of $G$, and a colouring $\kappa: V(G) \rightarrow Var(f)$, we introduce the following definitions:
\begin{itemize}
    \item $B(\gamma)$ is the set of bidirected edges of $\gamma$.
    \item $N(G, \gamma, \kappa)$ is the graph $(V^-(G), B(\gamma)|_{V^{-}(G)})$, where $V^{-}(G)$ is the set of vertices coloured with a negatively signed variable by $\kappa$ (and the edge set consists only of the bidirected edges from $\gamma$).
\end{itemize}
Also recall that for a graph $G$, we let $c(G)$ denote the number of connected components of $G$. 

\begin{theorem}\label{thm:xbpleth}
$$
XB_{(G,w)}[f] = \sum_{(\gamma, \kappa)} (1+t)^{|B(\gamma)|}(-1)^{c(N(G,\gamma,\kappa))}\left(\prod_{v \in V(G)} \kappa(v)^{w(v)}\right)
$$
where the sum runs over all acyclic biorientations $\gamma$ of $G$ and colourings $\kappa: V(G) \rightarrow Var(f)$, such that for each edge $uv \in E(G)$
\begin{itemize}
    \item If $\kappa(u) < \kappa(v)$, then the edge $uv$ is singly directed from $u$ to $v$.
    \item If $\kappa(u) = \kappa(v)$ is positive, then the edge $uv$ is bidirected.
    \item If $\kappa(u) = \kappa(v)$ is negative, then the edge $uv$ may be oriented in any of the three possible ways.
\end{itemize}
\end{theorem}

\begin{proof}

It is straightforward to verify that for any weighted graph $(G,w)$, we have $XB_{(G,w)}[1] = (1+t)^{|E(G)|}$ and $XB_{(G,w)}[\bx] = \sum_{\pi \vdash V(G)} (1+t)^{e(\pi)}\tm_{\lm(\pi)}$, and that both sides of the statement agree in these cases.

As before, we proceed by induction. Let $G_{\lm}$ be the graph with vertices of weights $\lm_1, \dots, \lm_k$ and no edges. Then $XB_{G_{\lm}} = X_{G_{\lm}} = p_{\lm}$, so the proofs for $XB$ all follow identically from the corresponding proofs for the chromatic symmetric function.

We proceed to the inductive step. Let our claim be true for all graphs with no more than $k-1$ edges for some $k \geq 1$, and we will establish the statement for graphs with $k$ edges.

Select any edge $e = v_1v_2 \in E(G)$. We use the deletion-contraction relation of $XB$ on $e$, fixing $\gamma$ an acyclic biorientation of $G-e$ and $\kappa: V(G) \rightarrow Var(f)$ that is compatible with it, and retaining the notation as in the proof of Theorem \ref{thm:xpleth}, with the addition of $a_G(\gamma_b, \kappa)$ for the case where we add a bidirected edge between $v_1$ and $v_2$, we wish to show that

\begin{equation}\label{eq:gammaxb}
a_G(\gamma_1, \kappa)+a_G(\gamma_2, \kappa)+a_G(\gamma_b, \kappa) = a_{G-e}(\gamma, \kappa)+t \cdot a_{G/e}(\gamma_e, \kappa_e).
\end{equation}

For a given graph $G$, and choice of an orientation $o$ and colouring $k$, let $N(G,o,k)$ denote the induced subgraph of negative-coloured vertices with only bidirected edges. Note that by definition $a_G(o,k)$ depends on $N(G,o,k)$.

First, we consider the case $v_1 = v_2$, so $e$ is a loop. Then with respect to $G$, neither $\gamma_1$ nor $\gamma_2$ are valid, but $\gamma_b$ is valid, and $a_G(\gamma_b, \kappa) = (1+t)a_{G-e}(\gamma, \kappa)$. This is true even if $\kappa(v_1)$ is negative, since in this case adding the bidirected loop $e$ to $N(G-e,\gamma,\kappa)$ does not change its number of connected components.

Also, $a_{G/e}(\gamma_e, \kappa_e) = a_{G-e}(\gamma, \kappa)$, so in this case both sides of \eqref{eq:gammaxb} are equal to $(1+t)a_{G-e}(\gamma, \kappa)$ and the equation is satisfied. For the remainder of the proof, we can thus assume that $v_1$ and $v_2$ are distinct.

Next, we assume there is at least one path of bidirected edges between $v_1$ and $v_2$. This means that necessarily $\kappa(v_1) = \kappa(v_2)$. In this case, $\gamma_1$ and $\gamma_2$ are both invalid, and $a_G(\gamma_b, \kappa) = (1+t)a_{G-e}(\gamma, \kappa)$ regardless of whether $\kappa(v_1)$ is positive or negative, since in either case adding a bidirected edge does not change the number of connected components of $N(G-e,\gamma,\kappa)$.

Next, we assume that there is a directed path between $v_1$ and $v_2$ that contains at least one singly directed edge. There cannot be such paths both from $v_1$ to $v_2$ and from $v_2$ to $v_1$, so assume that there exists such a path from $v_1$ to $v_2$. Then $\gamma_2$, $\gamma_b$, and $\gamma_e$ all fail to be acyclic regardless of $\kappa$. For the others we split into cases:
\begin{itemize}
    \item If $\kappa(v_1) < \kappa(v_2)$, or if $\kappa(v_1) = \kappa(v_2)$ is negative, then $\gamma_1$ is valid and the equation is satisfied.
    \item If $\kappa(v_1) = \kappa(v_2)$ is positive, then with respect to $\gamma$ in $G-e$ every vertex along the directed path must have colour $\kappa(v_1)$. But one of these edges is singly directed, a contradiction, so this case never occurs.
\end{itemize}

We now consider the case in which $\gamma$ has no such path. We split into cases based on $\kappa(v_1)$ and $\kappa(v_2)$:
\begin{itemize}
    \item If these colours are different, without loss of generality assume $\kappa(v_1) < \kappa(v_2)$. Then only $\gamma_1$ and $\gamma$ are valid orientations, and clearly \eqref{eq:gammaxb} is satisfied.
    \item If $\kappa(v_1) = \kappa(v_2)$ is positive, then $\gamma_b$, $\gamma$, and $\gamma_e$ are valid orientations, $a_G(\gamma_b, \kappa) = (1+t)a_{G-e}(\gamma, \kappa)$, and $a_{G / e}(\gamma_e, \kappa) = t \cdot a_{G-e}(\gamma, \kappa)$, and \eqref{eq:gammaxb} is satisfied.
    \item If $\kappa(v_1) = \kappa(v_2)$ is negative, then all orientations are valid. We may easily see that $a_G(\gamma_1, \kappa) = a_G(\gamma_2, \kappa) = a_{G-e}(\gamma, \kappa)$. We have $a_G(\gamma_b, \kappa) = -(1+t)a_G(\gamma, \kappa)$ since it has an additional bidirected edge, and it also unified two components of $N(G-e,\gamma,\kappa)$ since by assumption the endpoints of $v_1v_2$ were not previously connected by a bidirected path. Furthermore, $t \cdot a_{G/e}(\gamma_e, \kappa_e) = -t \cdot a_{G-e}(\gamma, \kappa)$ since this time we removed a component of $N(G-e,\gamma,\kappa)$, and the equality holds.
\end{itemize}
Since our desired equation holds in all cases, the induction is completed, and our statement proved.

\end{proof}

Using this plethystic interpretation, we may easily generalize properties of the chromatic symmetric function to the Tutte symmetric function. For example, the following trivially follows from the same proof as of Lemma \ref{lem:graphadd}:

\begin{cor}\label{cor:graphaddxb}

$$
XB_{(G,w)}[f+g] = \sum_{A \sqcup B = V(G)} XB_{(G|_A, w|_A)}[f]XB_{(G|_B,w|_B)}[g]
$$
where $A$ and $B$ form a partition of $V(G)$, and $w|_A$ is the restriction of the weight function $w$ to only those vertices in $A$.
\end{cor}

We also obtain an interpretation for the antipode of $XB_{(G,w)}$:
\begin{cor}\label{cor:antipodexb}
$$
XB_{(G,w)}[-\bx] = \sum_{S \subseteq E(G)} (1+t)^{|S|}X_{(G/S,w/S)}[-\bx]
$$
\end{cor}
\begin{proof}
We split the sum in Theorem \ref{thm:xbpleth} by the set $S$ of bidirected edges. Once $S$ is fixed, note that each connected component of $(V(G),S)$ must be monochromatic; then it is easy to see that each choice of $(\gamma, \kappa)$ such that $\gamma$'s bidirected edges are exactly $S$ are in one-to-one correspondence with choices of $(\gamma, \kappa)$ on $(G/S,w/S)$ admitting no bidirected edges in the same manner as in the proof of Lemma \ref{lem:graphadd}.
\end{proof}

\section{Further Directions}

As noted in the introduction, it is a major open question to evaluate general plethysms of the form $s_{\lm}[s_{\mu}]$. There are many known ways to expand the input $s_{\mu}$ so as to provide an interesting variable set for coloring \cite{stanleybook}. Although in general $s_{\lm}$ cannot be represented as the chromatic symmetric function of a single vertex-weighted graph (see \cite{choclass} and \cite[Footnote 1]{delcon}), Schur-function basis expansions of chromatic symmetric functions are an ongoing topic of research \cite{epos, wang}. Additionally, it would still be helpful if $s_{\lm}$ could be represented as a linear combination of chromatic symmetric functions in some natural way, and to the best of the authors' knowledge this is a topic that has not been previously explored.

Given this graph-theoretic interpretation for plethysm, it is natural to consider if related operations and constructions in algebraic combinatorics can be represented graph-theoretically by further extension. It was already known that unicellular LLT polynomials are chromatic quasisymmetric functions of unit interval graphs \cite{huhmelt}; recent work of Tom \cite{tom, tom2} has gone further, interpreting horizontal-strip LLT polynomials in terms of colorings and orientations of certain vertex-weighted graphs in a manner similar to that of Theorem \ref{thm:xpleth}. It seems very likely that continuing in a direction unifying these approaches would continue to shed light on LLT polynomials and their positivity conjectures. 

Finally, there are ways of defining modified Macdonald polynomials, through plethysm \cite[Section 3]{descouens} and through fillings of tableaux or similar objects \cite{hhl, colorful}. It would be interesting to determine if there is a graph-theoretic interpretation for Macdonald polynomials and related functions that provides another avenue for investigation.

\section{Acknowledgments}

We would like to thank Foster Tom for helpful discussions. We would like to thank Darij Grinberg and the anonymous referee for helpful suggestions that provided further information and greatly improved the clarity of this work in many places.

We acknowledge the support of the Natural Sciences and Engineering Research Council of Canada (NSERC), [funding
reference number RGPIN-2020-03912]. Cette recherche a \'et\'e financ\'ee par le Conseil de recherches en sciences naturelles et en g\'enie du Canada (CRSNG),
[num\'ero de r\'ef\'erence RGPIN-2020-03912].

\bibliographystyle{plain}
\bibliography{bib}

\end{document}